\documentclass[11pt]{amsart}%
\usepackage{amsmath}
\usepackage{mathrsfs}
\usepackage{amssymb}
\usepackage{antonioledda}
\usepackage{amsthm}
\usepackage{enumerate}
\usepackage{accents}
\usepackage{color}
\usepackage{rotating}
\usepackage{amsfonts}
\usepackage{graphicx}
\usepackage{hyperref}%
\providecommand{\U}[1]{\protect\rule{.1in}{.1in}}

\newtheorem{definition}{Definition}[section]

\theoremstyle{remark}

\numberwithin{equation}{section}
\newtheorem{theorem}{Theorem}[section]
\newtheorem{lemma}{Lemma}[section]
\newtheorem{proposition}{Proposition}[section]
\newtheorem{corollary}{Corollary}[section]
\theoremstyle{plain}
\begin{document}
\title{On some properties of PBZ*-lattices}
\author{Roberto Giuntini, Antonio Ledda, Francesco Paoli}
\address{University of Cagliari}

\begin{abstract}
We continue the algebraic investigation of \emph{PBZ*-lattices}, a notion
introduced in \cite{GLP1+} in order to obtain insights into the structure of
certain algebras of effects of a Hilbert space, lattice-ordered under the
spectral ordering.


\end{abstract}
\maketitle

\section{Introduction}

In \cite{GLP1+} we introduced the variety {$\mathbb{PBZL}^{\ast}$ of
}\emph{PBZ*-lattices} as an abstract algebraic counterpart of the structure%
\[
\left\langle \mathcal{E}\left(  \mathbf{H}\right)  ,\wedge,\vee,^{\prime
},^{\sim},\mathbb{O},\mathbb{I}\right\rangle \text{,}%
\]
where:

\begin{itemize}
\item $\mathcal{E}\left(  \mathbf{H}\right)  $ is the set of all effects of a
given complex separable Hilbert space;

\item $\wedge$ and $\vee$ are the meet and the join, respectively, of the
\emph{spectral ordering} $\leq_{s}$ so defined\footnote{That the spectral
ordering is indeed a lattice ordering has been essentially shown by Olson
\cite{Ols} and de Groote \cite{deG}, who also proved that it coincides with
the more familiar ordering of effects induced via the trace functional when
both orderings are restricted to the set of projection operators of the same
Hilbert space. The same ordering has also been given an algebraic treatment,
in a different context, in \cite{Dvur}.} for all $E,F\in\mathcal{E}\left(
\mathbf{H}\right)  $:
\[
E\leq_{s}F\,\,\,\text{iff}\,\,\,\forall\lambda\in\mathbb{R}:\,\,M^{F}%
(\lambda)\leq M^{E}(\lambda),
\]
where for any effect $E$, $M^{E}$ is the unique spectral family \cite[Ch.
7]{Kr} such that $E=\int_{-\infty}^{\infty}\lambda\,dM^{E}(\lambda)$ (the
integral is here meant in the sense of norm-converging Riemann-Stieltjes sums
\cite[Ch. 1]{Strocco});

\item $\mathbb{O}$ and $\mathbb{I}$ are the null and identity operators, respectively;

\item $E^{\prime}=\mathbb{I}-E$ and $E^{\sim}=P_{\ker\left(  E\right)  }$ (the
projection onto the kernel of $E$).
\end{itemize}

This class of algebras is further motivated, from a physical viewpoint, by the
fact that it reproduces at an abstract level the \textquotedblleft
collapse\textquotedblright\ of several notions of \emph{sharp physical
property} that is observed in the concrete physical model over $\mathcal{E}%
\left(  \mathbf{H}\right)  $; from an algebraic viewpoint, moreover, it can be
viewed as an unsharp generalisation of orthomodular lattices that also covers
certain expansions of Kleene lattices --- and, consequently, may have some
potential interest for many-valued logics of partial information
\cite{Bergman}.

In \cite{GLP1+}, we started an algebraic investigation of {$\mathbb{PBZL}%
^{\ast}$, which we continue in the present paper. In Section \ref{prel}, the
main definitions and results of }\cite{GLP1+} are summarised, with an eye to
making this work reasonably self-contained. In Section {\ref{porto}, we
improve on the description of the lattice of subvarieties of $\mathbb{PBZL}%
^{\ast}$ given in }the same paper. The main result of this section is a
representation theorem for subdirectly irreducible antiortholattices (see
below for a definition) in terms of twist structures over bounded lattices,
which is applied later on to identify a suitable set of generators for the
distributive subvariety of the variety $V\left(  \mathbb{AO}{\mathbb{L}%
}\right)  $ generated by antiortholattices. We also show that the distributive
subvariety of $V\left(  \mathbb{AO}{\mathbb{L}}\right)  $ satisfying the De
Morgan law for the intuitionistic complement $^{\sim}$ is generated by the
$5$-element antiortholattice chain, and we simplify an already known
equational basis for the subvariety of $V\left(  \mathbb{AO}{\mathbb{L}%
}\right)  $ generated by the $3$-element antiortholattice chain. In Section
{\ref{orizzo} we investigate horizontal sums of PBZ*-lattices, showing that
the variety generated by }subdirectly irreducible {PBZ$^{\ast}$-lattices}
which are a horizontal sum of Boolean algebras and an ortholattice chain is
generated by its finite members\footnote{We acknowledge here G. Cattaneo's
insightful comments on \cite{GLP1+}, which we received after the paper itself
had gone into its production stage.}.

\section{ Preliminaries\label{prel}}

We recap in this section some definition and results (the latter mostly from
\cite{GLP1+}, except when explicitly noted) which will be needed in the sequel.

\begin{definition}
\label{de:kleene}A bounded involution lattice $\mathbf{L}=\left\langle
L,\wedge,\vee,^{\prime},0,1\right\rangle $ is a \emph{pseudo-Kleene algebra}
in case it satisfies any of the following two equivalent conditions:

\begin{enumerate}
\item for all $a,b\in L$, $\,$if$\,\,a\leq a^{\prime}\,\,$and$\,\,b\leq
b^{\prime},\,\,$then$\,a\leq b^{\prime}\,$;

\item for all $a,b\in L$, $a\wedge a^{\prime}\leq b\vee b^{\prime}$.
\end{enumerate}
\end{definition}

The variety of pseudo-Kleene algebras, for which see e.g. \cite{PSK}, is
denoted by $\mathbb{PKA}$. Distributive pseudo-Kleene algebras are variously
called \emph{Kleene lattices} or \emph{Kleene algebras} in the literature.
Observe that in \cite{GLP1+}, embracing the terminological usage from \cite[p.
12]{RQT}, pseudo-Kleene algebras were referred to as \textquotedblleft Kleene
lattices\textquotedblright. Here, we switch to the less ambiguous
\textquotedblleft pseudo-Kleene algebras\textquotedblright, following the
suggestion of a referee.

In unsharp quantum logic, there are several competing purely algebraic
characterisations of sharp effects \cite[Ch. 7]{RQT}. A quantum effect or
property is usually called \emph{sharp} if it satisfies the noncontradiction principle:

\begin{definition}
\label{de:kleenesharp}Let $\mathbf{L}\in\mathbb{PKA}$.

\begin{enumerate}
\item An element $a\in L$ is said to be \emph{Kleene-sharp} iff $a\wedge
a^{\prime}=0$. $S_{K}(L)\,\,$denotes the class of Kleene-sharp elements
of$\,\,\mathbf{L}$.

\item An \emph{ortholattice\/} is a bounded involution lattice $\mathbf{L}$
such that $S_{K}(L)=L$. The variety of ortholattices is denoted by
$\mathbb{OL}$.
\end{enumerate}
\end{definition}

Among ortholattices, \emph{orthomodular lattices }play a crucial role in the
standard (sharp) approach to quantum logic. Recall that an orthomodular
lattice is an ortholattice $\mathbf{L}$ such that, for all $a,b\in L$,
$\,$if$\,\,\,a\leq b,\,\,\,$then$\,\,b=(b\wedge a^{\prime})\vee a$. The class
of orthomodular lattices is actually a variety, hereafter denoted by
$\mathbb{OML}$.

It is well-known that an ortholattice $\mathbf{L}$ is orthomodular if and only
if, for all $a,b\in L$, $\,$if$\,\,a\leq b\,\,\,$and$\,\,a^{\prime}\wedge
b=0,\,\,$then$\,\,a=b$. The right-to-left direction of this equivalence fails
in the wider setting of bounded involution lattices. This justifies the
following definition, aimed at recapturing at least in part the force of the
orthomodular condition in bounded involution lattices that may contain
Kleene-unsharp elements.

\begin{definition}
\label{de:paraorthomodular}An algebra $\mathbf{L}$ with a bounded involution
lattice term reduct is said to be \emph{paraorthomodular} iff, for all $a,b\in
L$:
\[
\,\text{if}\,\,a\leq b\,\,\,\text{and}\,\,a^{\prime}\wedge b=0,\,\,\text{then}%
\,\,a=b.
\]

\end{definition}

It turns out that the class of paraorthomodular pseudo-Kleene algebras is a
proper quasivariety, whence we cannot help ourselves to the strong universal
algebraic properties that characterise varieties. It is then natural to wonder
whether there exists an \emph{expansion} of the language of bounded involution
lattices where the paraorthomodular condition can be equationally recovered.
The appropriate language expansion is provided by including an additional
unary operation and moving to the $\left\langle 2,2,1,1,0,0\right\rangle $
type, already familiar to unsharp quantum logicians and algebraists from the
investigation of \emph{Brouwer-Zadeh lattices} (see \cite{CN} or \cite[Ch.
4.2]{RQT}).

\begin{definition}
\label{de:bazar}\noindent A \emph{Brouwer Zadeh lattice\/} (or
\emph{BZ-lattice\/}) is an algebra%
\[
\mathbf{B}=\langle{B\,,\,\wedge,\vee\,,\,}^{\prime}{,\,^{\sim},0\,,1}\rangle
\]
of type $\left\langle 2,2,1,1,0,0\right\rangle $, such that:

\begin{enumerate}
\item $\langle{B\,,\,\wedge,\vee\,,\,}^{\prime}{,0\,,1}\rangle$ is a
pseudo-Kleene algebra;

\item for all $a,b\in B$, the following conditions are satisfied:
\[%
\begin{array}
[c]{ll}%
\text{(i) }a\wedge a^{\sim}={0}\text{;} & \text{(ii) }a\leq a^{\sim\sim}\\
\text{(iii) }a\leq b\ \,\text{implies}\,\ b^{\sim}\leq a^{\sim}\text{;} &
\text{(iv) }a{^{\sim\prime}=}a^{\sim\sim}\text{.}%
\end{array}
\]

\end{enumerate}
\end{definition}

The class of all BZ-lattices is a variety, denoted by $\mathbb{BZL}$;
$\mathbb{OL}$ can be identified with the subvariety of $\mathbb{BZL}$ whose
relative equational basis w.r.t. $\mathbb{BZL}$ is given by the equation
$x{\,^{\sim}=x}^{\prime}$. In any BZ-lattice, we set $\Diamond x=x^{\sim\sim}$
and $\Box x=x^{\prime\sim}$. The following arithmetical lemma, the proof of
which is variously scattered in the above-mentioned literature and elsewhere,
will be used with no special mention in what follows.

\begin{lemma}
\label{basics}Let $\mathbf{L}$ be a BZ-lattice. For all $a,b\in L$, the
following conditions hold:

\begin{description}
\item[(i)] $a^{\sim\sim\sim}=a^{\sim}$;

\item[(ii)] $a^{\sim}\leq a^{\prime}$;

\item[(iii)] $(a\vee b)^{\sim}=a^{\sim}\wedge b^{\sim}$;

\item[(iv)] $a^{\sim}\vee b^{\sim}\leq(a\wedge b)^{\sim}$;

\item[(v)] $(\Box(a^{\prime}))^{\prime}=\Diamond a$;

\item[(vi)] $\Box(a\wedge b)=\Box a\wedge\Box b$;

\item[(vii)] $\Diamond(a\vee b)=\Diamond a\vee\Diamond b$;

\item[(viii)] $\Diamond(a\wedge b)\leq\Diamond a\wedge\Diamond b$;

\item[(ix)] if $a^{\prime}\leq a$, then $a^{\sim}=0$.
\end{description}
\end{lemma}

We remarked above that Kleene-sharpness is not the unique purely algebraic
characterisation of a sharp quantum property. Two noteworthy alternatives now
become available in our expanded language of BZ-lattices.

\begin{definition}
\label{de:modalbrouwer} Let $\mathbf{L}$ be a BZ-lattice.

\begin{enumerate}
\item An element $a\in L$ is said to be \emph{$\Diamond$-sharp} iff
$a=\Diamond a$ (equivalently, iff $a^{\prime}=a^{\sim}$); the class of all
$\Diamond$-sharp elements of $\mathbf{L}$ will be denoted by $S_{\Diamond}(L)$.

\item An element $a\in L$ is said to be \emph{Brouwer-sharp} iff $a\vee
a^{\sim}=1$; the class of all Brouwer-sharp elements of $\mathbf{L}$ will be
denoted by $S_{B}(L)$.
\end{enumerate}
\end{definition}

For a generic BZ-lattice $\mathbf{L}$, we have that $S_{\Diamond}(L)\subset
S_{B}(L)\subset S_{K}(L)$. However, in any BZ-lattice of effects of a Hilbert
space (under the meet and join operation induced by the spectral ordering)
these three classes coincide. Consequently, it makes sense to investigate
whether there is a class of BZ-lattices for which this collapse result can be
recovered at a purely abstract level. The next definition and theorem answer
this question in the affirmative.

\begin{definition}
\label{de:bzlstar}A \emph{BZ$^{\ast}$-lattice} is a BZ-lattice $\mathbf{L}$
that satisfies the following condition for all $a\in L:$%
\[
\left(  \ast\right)  \noindent\qquad\left(  a\wedge a^{\prime}\right)  ^{\sim
}\leq a^{\sim}\vee a^{\prime\sim}.
\]

\end{definition}

\begin{theorem}
\label{co:collassone} Let $\mathbf{L}$ be a paraorthomodular BZ$^{\ast}%
$-lattice. Then,
\[
S_{\Diamond}(L)=S_{B}(L)=S_{K}(L).
\]

\end{theorem}

As pleasing as this result may be, the class of paraorthomodular BZ$^{\ast}%
$-lattices still suffers from a major shortcoming: the paraorthomodularity
condition is quasiequational. However, the next result shows that it can be
replaced by an equation.

\begin{theorem}
\label{th:paradia} Let $\mathbf{L}$ be a BZ$^{\ast}$-lattice. The following
conditions are equivalent:

\begin{enumerate}
\item[(1)] $\mathbf{L}$ is paraorthomodular;

\item[(2)] $\mathbf{L}$ satisfies the following \emph{$\Diamond$%
-orthomodularity} condition for all $a,b\in L$:
\[
\,(a^{\sim}\vee(\Diamond a\wedge\Diamond b))\wedge\Diamond a\leq\Diamond b.
\]

\end{enumerate}
\end{theorem}

Such a variety will be denoted by $\mathbb{PBZL}^{\ast}$, and its members will
be referred to as \emph{PBZ}$^{\ast}$\emph{-lattices}. {It can be seen that
every bounded lattice can be embedded as a sublattice into a PBZ*-lattice.
Consequently, $\mathbb{PBZL}^{\ast}$ satisfies no nontrivial identity in the
language of lattices.}

The naturalness of this concept is further reinforced by the circumstance that
BZ-lattices of effects of a Hilbert space, under the spectral ordering,
qualify as instances of PBZ$^{\ast}$-lattices:

\begin{theorem}
{\ \label{concreto}Let $\mathbf{H}$ be a Hilbert space. The algebra%
\[
\mathbf{E}\left(  \mathbf{H}\right)  =\left\langle \mathcal{E}\left(
\mathbf{H}\right)  ,\wedge_{s},\vee_{s},^{\prime},^{\sim},\mathbb{O}%
,\mathbb{I}\right\rangle ,
\]
(see the introduction for the notation) is a PBZ*-lattice. Moreover,
$S_{K}(\mathbf{E}\left(  \mathbf{H}\right)  )=$}$S_{\Diamond}(\mathbf{E}%
\left(  \mathbf{H}\right)  )=S_{B}(\mathbf{E}\left(  \mathbf{H}\right)
)${\ is an orthomodular subuniverse of $\mathbf{E}\left(  \mathbf{H}\right)  $
consisting of all the projection operators of $\mathbf{H}$. }
\end{theorem}

All orthomodular lattices are, of course, {PBZ*-lattices. }Further examples of
{PBZ*-lattices are given by those algebras in this class which are
\textquotedblleft as far apart as possible\textquotedblright\ from
orthomodular lattices. In any orthomodular lattice }${\mathbf{L}}${,
$S_{K}(L)=L$; on the other hand, $\mathbf{L}$ is an \emph{antiortholattice} if
$S_{K}(L)=\left\{  0,1\right\}  $. We denote by $\mathbb{AOL}$ the class of
antiortholattices. For all }$n\geq1$, the $n$-element Kleene chain
{$\mathbf{D}_{n}$, such that }$a^{\sim}=0$ iff $a>0${, is an antiortholattice.
}

{Some elementary properties of antiortholattices follow. }

\begin{lemma}
{\ }

\begin{enumerate}
\item {$\mathbf{L}\in$ $\mathbb{AOL}$ iff, for all }$a\in L$, $a^{\sim}=0$ if
$a>0$ and $a^{\sim}=1$ otherwise;

\item {Every $\mathbf{L}\in$ $\mathbb{AOL}$ is directly indecomposable;}

\item {$\mathbb{AOL}$ is a proper universal class. }
\end{enumerate}
\end{lemma}

An easy observation that paves the way for the study of the lattice
{$\mathbf{L}_{\mathbb{PBZL}^{\ast}}$ of subvarieties of $\mathbb{PBZL}^{\ast}
$ is the fact that the $2$-element Boolean algebra is a subalgebra of every
nontrivial PBZ*-lattice. Therefore, $\mathbf{L}_{\mathbb{PBZL}^{\ast}} $ has a
single atom, the variety $\mathbb{BA}$ of Boolean algebras. It is well-known
that $\mathbb{BA}$ has a single orthomodular cover \cite[Cor. 3.6]{BH}: the
variety $V\left(  \mathbf{MO}_{2}\right)  $, generated by the simple modular
ortholattice with $4$ atoms. Two questions naturally arise: how many
non-orthomodular covers of the Boolean variety are there? Is it possible to
give a finite equational basis for the variety $V\left(  \mathbb{AOL}\right)
$ generated by antiortholattices? The next theorem provides a solution to
these problems.}

\begin{theorem}
\label{caniggia}

\begin{enumerate}
\item {An equational basis for $V\left(  \mathbb{AOL}\right)  $ relative to
$\mathbb{PBZL}^{\ast}$ is given by the identities%
\begin{align*}
\text{(AOL1) }  &  \left(  x^{\sim}\vee y^{\sim}\right)  \wedge\left(
\Diamond x\vee z^{\sim}\right)  \approx\left(  \left(  x^{\sim}\vee y\right)
\wedge\left(  \Diamond x\vee z\right)  \right)  ^{\sim}\text{;}\\
\text{(AOL2) }  &  x\approx\left(  x\wedge y^{\sim}\right)  \vee\left(
x\wedge\Diamond y\right)  \text{;}\\
\text{(AOL3) }  &  x\approx\left(  x\vee y^{\sim}\right)  \wedge\left(
x\vee\Diamond y\right)  \text{.}%
\end{align*}
}

\item {\ $\mathbb{OML}\cap V\left(  \mathbb{AOL}\right)  =\mathbb{BA}$};

\item There is a single non-orthomodular cover of {$\mathbb{BA}$, the variety
}$V\left(  {\mathbf{D}_{3}}\right)  $ generated by the $3$-element
antiortholattice chain.
\end{enumerate}
\end{theorem}

We recall from \cite{CGP} that $V\left(  \mathbf{D}_{3}\right)  $ is
axiomatised relative to $V\left(  \mathbb{AOL}\right)  $ by the identities%
\begin{align*}
\text{(DIST) }x\wedge\left(  y\vee z\right)   &  \approx\left(  x\wedge
y\right)  \vee\left(  x\wedge z\right)  ;\\
\text{(SDM)\ }\left(  x\wedge y\right)  ^{\sim}  &  \approx x^{\sim}\vee
y^{\sim};\\
\text{(SK) }x\wedge\Diamond y  &  \leq\square x\vee y.
\end{align*}

\section{More on the lattice of PBZ* varieties\label{porto}}

In this section, we intend to explore in greater detail the structure {of
$\mathbf{L}_{\mathbb{PBZL}^{\ast}}$, extending the preliminary results on its
structure contained in }\cite{GLP1+}. In particular, we will focus on the
antiortholattice side of {$\mathbf{L}_{\mathbb{PBZL}^{\ast}}$ and try to
investigate some properties of }$V\left(  \mathbb{AOL}\right)  $ and some of
its notable subvarieties. The only exception to this policy is the next lemma,
which disproves the natural conjecture to the effect that the join of the
orthomodular variety $\mathbb{OML}$ and the antiortholattice variety $V\left(
\mathbb{AOL}\right)  $ covers the whole of $\mathbb{PBZL}^{\mathbb{\ast}}$.

\begin{lemma}
$V\left(  \mathbb{AOL}\right)  \vee\mathbb{OML}<\mathbb{PBZL}^{\mathbb{\ast}}
$.
\end{lemma}

\begin{proof}
The identity%
\[
x\vee y\approx\left(  \left(  x\vee y\right)  \wedge y^{\sim}\right)
\vee\left(  \left(  x\vee y\right)  \wedge\lozenge y\right)
\]
is readily seen to hold both in $\mathbb{AOL}$ and in $\mathbb{OML}$. However,
it fails in $\mathbb{PBZL}^{\mathbb{\ast}}$. In fact, consider the
PBZ*-lattice displayed in the next figure:%
\begin{center}
[Missing Picture]
\end{center}
with $a^{\sim}=b^{\prime\sim}=b $, $b^{\sim}=b^{\prime}$ and $a^{\prime\sim
}=c^{\sim}=0$. Our conclusion follows upon observing that%
\[
c\vee a=c\neq a=\left(  \left(  c\vee a\right)  \wedge a^{\sim}\right)
\vee\left(  \left(  c\vee a\right)  \wedge\lozenge a\right)  \text{.}%
\]

\end{proof}

Let us now comfortably settle within the boundaries of $V\left(
\mathbb{AOL}\right)  $. For a start, consider the next three varieties:
\begin{align*}
\mathbb{V}_{1}  &  =Mod\left(  Eq\left(  \mathbb{AOL}\right)  \cup\left\{
\text{DIST}\right\}  \right)  \text{;}\\
\mathbb{V}_{2}  &  =Mod\left(  Eq\left(  \mathbb{AOL}\right)  \cup\left\{
\text{SDM}\right\}  \right)  \text{;}\\
\mathbb{V}_{3}  &  =Mod\left(  Eq\left(  \mathbb{AOL}\right)  \cup\left\{
\text{DIST, SDM}\right\}  \right)  \text{.}%
\end{align*}

\begin{lemma}
(i) $V\left(  \mathbf{D}_{3}\right)  <\mathbb{V}_{3}$; (ii) $\mathbb{V}_{1}$
and $\mathbb{V}_{2}$ are incomparable varieties, hence strictly smaller than
$V\left(  \mathbb{AOL}\right)  $.
\end{lemma}

\begin{proof}
(i) It is easy to check that the algebra $\mathbf{D}_{4}$ belongs to
$\mathbb{V}_{3}$. Now, let $a,a^{\prime}$ denote the elements of $D_{4}$
different from the bounds, s.t. $a<a^{\prime}$. We have that $a^{\prime
}=a^{\prime}\wedge\Diamond a>\square a^{\prime}\vee a=a$, whence
$\mathbf{D}_{4}$ fails SK.

(ii) The antiortholattice%
\begin{center}
[Missing Picture]
\end{center}
is a non-distributive algebra in $\mathbb{V}_{2}$. The next distributive
antiortholattice:%
\begin{center}
[Missing Picture]
\end{center}
fails SDM because $\left(  a\wedge b\right)  ^{\sim}=0^{\sim}=1>0=0\vee
0=a^{\sim}\vee b^{\sim}$.
\end{proof}

\subsection{The variety generated by antiortholattices}

Next, we prove a representation theorem for the subdirectly irreducible
algebras in $V(\mathbb{AOL})$. We intend to put to good use a certain kind of
\emph{twist construction}, along the lines of what is done in \cite{Kalman,
Cignoli, PSK, GLP1+} and in several other papers on pseudo-Kleene algebras and
related structures.

Let $\mathbf{L}$ be a PBZ*-lattice. We define:%
\begin{align*}
N\left(  \mathbf{L}\right)   &  =\left\{  a\in L:a\leq a^{\prime}\right\}
\text{ (the set of \emph{negative} elements);}\\
P\left(  \mathbf{L}\right)   &  =\left\{  a\in L:a^{\prime}\leq a\right\}
\text{ (the set of \emph{positive} elements);}\\
N\left(  \mathbf{L}\right)  ^{+}  &  =\left\{  a\in L:a<a^{\prime}\right\}
\text{ (the set of \emph{strictly negative} elements);}\\
P\left(  \mathbf{L}\right)  ^{+}  &  =\left\{  a\in L:a^{\prime}<a\right\}
\text{ (the set of \emph{strictly positive} elements);}%
\end{align*}

By the pseudo-Kleene quasiequation, $N\left(  \mathbf{L}\right)  $ and
$P\left(  \mathbf{L}\right)  $ are subuniverses of the lattice reduct of
$\mathbf{L}$. With a slight notational abuse, we continue to denote by the
same symbols $N\left(  \mathbf{L}\right)  $ and $P\left(  \mathbf{L}\right)
$, respectively, the corresponding lattice subreducts of $\mathbf{L}$.

\begin{lemma}
\label{gustoso}Let $\mathbf{L}\in V(\mathbb{AOL})$ be a subdirectly
irreducible algebra. Then:

\begin{enumerate}
\item $\mathbf{L}$ is an antiortholattice;

\item $P\left(  {\mathbf{L}}\right)  \cup N\left(  {\mathbf{L}}\right)  =L$.
\end{enumerate}
\end{lemma}

\begin{proof}
(1) If $\mathbf{L}$ is s.i., then it is directly indecomposable, whence its
central elements are $0,1$. Since $\mathbf{L}$ is in $V(\mathbb{AOL})$, its
central elements coincide with its sharp elements. So $S_{K}\left(
\mathbf{L}\right)  =\left\{  0,1\right\}  $ and $\mathbf{L}$ is an antiortholattice.

(2) By (1), we can safely assume that $\mathbf{L}$ is an antiortholattice.
Suppose ex absurdo that there exists $a$ such that $a\nleq a^{\prime}$ and
$a^{\prime}\nleq a$. Let $\theta_{1}$ be the equivalence whose unique
non-singleton classes are the intervals $\left[  a\wedge a^{\prime},a\right]
$ and $\left[  a^{\prime},a\vee a^{\prime}\right]  $; likewise, let
$\theta_{2}$ be the equivalence whose unique non-singleton classes are the
intervals $\left[  a\wedge a^{\prime},a^{\prime}\right]  $ and $\left[
a,a\vee a^{\prime}\right]  $. Indeed, we have that $\theta_{1},\theta
_{2}>\Delta$: in fact, if $\left[  a\wedge a^{\prime},a\right]  $ were a
singleton, then $a\leq a^{\prime}$, against our hypothesis, and similarly for
the remaining intervals.

We now show that $\theta_{1}$ is a congruence. Disregarding trivial cases, if
$\left(  b_{1},c_{1}\right)  ,\left(  b_{2},c_{2}\right)  \in\theta_{1}$, then
we have to distinguish different situations. If $a\wedge a^{\prime}\leq
b_{1},c_{1},b_{2},c_{2}\leq a$, then clearly%
\[
a\wedge a^{\prime}\leq b_{1}\wedge b_{2},c_{1}\wedge c_{2}\leq a
\]
and thus $\left(  b_{1}\wedge b_{2},c_{1}\wedge c_{2}\right)  \in\theta_{1}$.
If $a\wedge a^{\prime}\leq b_{1},c_{1}\leq a$ and $a^{\prime}\leq b_{2}%
,c_{2}\leq a\vee a^{\prime}$, then $a\wedge a^{\prime}\leq a^{\prime}\leq
b_{2},c_{2}$, and thus%
\[
a\wedge a^{\prime}\leq b_{1}\wedge b_{2}\leq b_{1}\leq a\text{ and }a\wedge
a^{\prime}\leq c_{1}\wedge c_{2}\leq c_{1}\leq a
\]
and the same conclusion follows. The other cases are disposed of similarly,
whence $\theta_{1}$ preserves meets. It is readily seen that if $b,c\in\left[
a\wedge a^{\prime},a\right]  $, then $b^{\prime},c^{\prime}\in\left[
a^{\prime},a\vee a^{\prime}\right]  $, and vice versa. Finally, our hypothesis
that $a\nleq a^{\prime}$ and $a^{\prime}\nleq a$ implies that $0<a,a^{\prime}%
$; if it were that $a\wedge a^{\prime}=0$, then $a$ would be a nonzero sharp
element, against the fact that $\mathbf{L}$ is an antiortholattice. In sum, if
either $b,c\in$ $\left[  a\wedge a^{\prime},a\right]  $ or $b,c\in\left[
a^{\prime},a\vee a^{\prime}\right]  $, then $0<b,c$ and then $b^{\sim}%
=c^{\sim}=0$. Similarly, it can be shown that $\theta_{2}$ is a congruence.

To obtain a contradiction, it remains to be proved that $\theta_{1}\cap
\theta_{2}=\Delta$. Again, we have to split cases. If $a\wedge a^{\prime}\leq
b,c\leq a$ and $a\wedge a^{\prime}\leq b,c\leq a^{\prime}$, then $a\wedge
a^{\prime}\leq b,c\leq a\wedge a^{\prime}$, whence $b=a\wedge a^{\prime}=c$.
If $a\wedge a^{\prime}\leq b,c\leq a$ and $a\leq b,c\leq a\vee a^{\prime}$,
then%
\[
a\wedge a^{\prime}\leq b,c\leq a\leq b,c\text{,}%
\]
whereby, again, $b=a=c$. The remaining cases yield similar outcomes.
\end{proof}

We now introduce a slight variant of the twist construction already used in
\cite{GLP1+}. {Let $\mathbf{L=}\left\langle L,\wedge,\vee\right\rangle $ be a
bounded lattice with bounds $0$ and $1$ and induced order $\leq$. Take a
bijective copy $f\left[  L-\left\{  0\right\}  \right]  $ of $L-\left\{
0\right\}  $ and endow it with the order $\leq^{\partial}$dual to $\leq$; call
the resulting lattice $\mathbf{M}$. Now, consider the \emph{ordinal sum}
$\mathbf{M\oplus L}$ \cite[Ch. 4]{Harz}, which is a lattice. Let, for $a\in
M\oplus L$,%
\[
a^{\prime}=\left\{
\begin{array}
[c]{l}%
f\left(  a\right)  \text{ if }a\in L-\left\{  0\right\}  \text{;}\\
0\text{ if }a=0\text{;}\\
f^{-1}\left(  a\right)  \text{ if }a\in M\text{.}%
\end{array}
\right.
\]
The algebra }$\mathcal{T}_{1}\left(  {\mathbf{L}}\right)  ${$=\left\langle
M\oplus L,\wedge^{\mathcal{T}_{1}\left(  {\mathbf{L}}\right)  },\vee
^{\mathcal{T}_{1}\left(  {\mathbf{L}}\right)  },^{\prime\mathcal{T}_{1}\left(
{\mathbf{L}}\right)  },^{\sim\mathcal{T}_{1}\left(  {\mathbf{L}}\right)
},0^{\mathcal{T}_{1}\left(  {\mathbf{L}}\right)  },1^{\mathcal{T}_{1}\left(
{\mathbf{L}}\right)  }\right\rangle $, where:%
\[%
\begin{array}
[c]{lll}%
\wedge^{\mathcal{T}_{1}\left(  {\mathbf{L}}\right)  }=\wedge^{\mathbf{M\oplus
L}}\text{;} & \vee^{\mathcal{T}_{1}\left(  {\mathbf{L}}\right)  }%
=\vee^{\mathbf{M\oplus L}}\text{;} & 0^{\mathcal{T}_{1}\left(  {\mathbf{L}%
}\right)  }=f\left(  1\right)  \text{;}\\
^{\prime\mathcal{T}_{1}\left(  {\mathbf{L}}\right)  }=^{\prime}\text{;} &  &
1^{\mathcal{T}_{1}\left(  {\mathbf{L}}\right)  }=1
\end{array}
\]
}and{%
\[
a^{\sim\mathcal{T}_{1}\left(  {\mathbf{L}}\right)  }=\left\{
\begin{array}
[c]{l}%
1\text{ if }a=0^{\mathcal{T}_{1}\left(  {\mathbf{L}}\right)  }\text{;}\\
0^{\mathcal{T}_{1}\left(  {\mathbf{L}}\right)  }\text{ otherwise.}%
\end{array}
\right.
\]
is an antiortholattice such that }$P\left(  \mathcal{T}_{1}\left(
{\mathbf{L}}\right)  \right)  \cup N\left(  \mathcal{T}_{1}\left(
{\mathbf{L}}\right)  \right)  =M\oplus L${, and $\mathbf{L}$ is a lattice
subreduct of such --- actually, }$L=P\left(  \mathcal{T}_{1}\left(
{\mathbf{L}}\right)  \right)  ${. A similar construction can be effected for
any \emph{upper bounded} lattice $\mathbf{L}$, with the only difference that
we take a bijective copy $f\left[  L\right]  $ of the whole universe $L$ of
the lattice at issue, not just of $L-\left\{  0\right\}  $. The resulting
PBZ*-lattice, which we call }$\mathcal{T}_{2}\left(  {\mathbf{L}}\right)  $,
will be such that:

\begin{itemize}
\item if {$\mathbf{L}$ has a bottom element }$0$, then $P\left(
\mathcal{T}_{2}\left(  {\mathbf{L}}\right)  \right)  $ has $0$ as a bottom
element and $N\left(  \mathcal{T}_{2}\left(  {\mathbf{L}}\right)  \right)  $
has $f\left(  0\right)  $ as a top element;

\item if {$\mathbf{L}$ is unbounded below}, then $P\left(  \mathcal{T}%
_{2}\left(  {\mathbf{L}}\right)  \right)  $ has no bottom element and
$N\left(  \mathcal{T}_{2}\left(  {\mathbf{L}}\right)  \right)  $ has no top element.
\end{itemize}

Conversely, we have that:

\begin{theorem}
\label{gazzarra}Every nontrivial antiortholattice ${\mathbf{L}}$ such that
$P\left(  {\mathbf{L}}\right)  \cup N\left(  {\mathbf{L}}\right)  =L$ is
isomorphic to $\mathcal{T}_{i}\left(  P\left(  {\mathbf{L}}\right)  \right)
$, for some $i\in\left\{  1,2\right\}  $.
\end{theorem}

\begin{proof}
Let $i$ be $1$ or $2$ according as there is $c\in L$ such that $c=c^{\prime} $
or not. In the former case, in fact, $c$ is the bottom element in $P\left(
{\mathbf{L}}\right)  $ and the construction can be applied. Let $\varphi
:L\rightarrow\mathcal{T}_{i}\left(  P\left(  \mathbf{L}\right)  \right)  $ be
defined by:%
\[
\varphi\left(  a\right)  =\left\{
\begin{array}
[c]{l}%
a\text{ if }a\in P\left(  {\mathbf{L}}\right)  \text{;}\\
f\left(  a^{\prime{\mathbf{L}}}\right)  \text{ otherwise.}%
\end{array}
\right.
\]

This function is clearly injective, and it is surjective because $P\left(
{\mathbf{L}}\right)  \cup N\left(  {\mathbf{L}}\right)  =L$ and $f$ is a
bijection. Moreover, $\varphi\left(  0^{{\mathbf{L}}}\right)  =f\left(
1^{{\mathbf{L}}}\right)  =0^{\mathcal{T}_{i}\left(  {\mathbf{L}}\right)  }$.

Let us check that $\varphi\left(  a\wedge^{{\mathbf{L}}}b\right)
=\varphi\left(  a\right)  \wedge^{\mathcal{T}_{i}\left(  P\left(  {\mathbf{L}%
}\right)  \right)  }\varphi\left(  b\right)  $. If $a\wedge^{{\mathbf{L}}}b\in
P\left(  {\mathbf{L}}\right)  $, then $a,b\in P\left(  {\mathbf{L}}\right)  $,
and our conclusion is trivial. On the other hand, if $a\wedge^{{\mathbf{L}}%
}b\in N\left(  \mathbf{L}\right)  ^{+}$, then suppose further that $a,b\in
N\left(  \mathbf{L}\right)  ^{+}$. Then%
\begin{align*}
\varphi\left(  a\wedge^{{\mathbf{L}}}b\right)   &  =f\left(  \left(
a\wedge^{{\mathbf{L}}}b\right)  ^{\prime{\mathbf{L}}}\right) \\
&  =f\left(  a^{\prime{\mathbf{L}}}\vee^{{\mathbf{L}}}b^{\prime{\mathbf{L}}%
}\right) \\
&  =f\left(  a^{\prime{\mathbf{L}}}\right)  \wedge^{\mathcal{T}_{i}\left(
{\mathbf{L}}\right)  }f\left(  b^{\prime{\mathbf{L}}}\right) \\
&  =\varphi\left(  a\right)  \wedge^{\mathcal{T}_{i}\left(  P\left(
{\mathbf{L}}\right)  \right)  }\varphi\left(  b\right)  \text{.}%
\end{align*}

Suppose instead w.l.g. that $a\in N\left(  \mathbf{L}\right)  ^{+},b\in
P\left(  {\mathbf{L}}\right)  $. Then%
\[
\varphi\left(  a\wedge^{{\mathbf{L}}}b\right)  =\varphi\left(  a\right)
=f\left(  a^{\prime{\mathbf{L}}}\right)  =f\left(  a^{\prime{\mathbf{L}}%
}\right)  \wedge^{\mathcal{T}_{i}\left(  {\mathbf{L}}\right)  }b=\varphi
\left(  a\right)  \wedge^{\mathcal{T}_{i}\left(  P\left(  {\mathbf{L}}\right)
\right)  }\varphi\left(  b\right)  \text{.}%
\]

Next, we verify that $\varphi\left(  a^{\prime{\mathbf{L}}}\right)
=\varphi\left(  a\right)  ^{\prime\mathcal{T}_{i}\left(  {\mathbf{L}}\right)
}$. If $a\in N\left(  \mathbf{L}\right)  ^{+}$, then $a^{\prime{\mathbf{L}}%
}\in P\left(  {\mathbf{L}}\right)  $ and so $\varphi\left(  a^{\prime
{\mathbf{L}}}\right)  =a^{\prime{\mathbf{L}}}=f^{-1}\left(  f\left(
a^{\prime{\mathbf{L}}}\right)  \right)  =f^{-1}\left(  \varphi\left(
a\right)  \right)  =\varphi\left(  a\right)  ^{\prime\mathcal{T}_{i}\left(
{\mathbf{L}}\right)  }$. If $a\in P\left(  \mathbf{L}\right)  ^{+}$, then
$a^{\prime{\mathbf{L}}}\in N\left(  \mathbf{L}\right)  ^{+}$ and so
$\varphi\left(  a^{\prime{\mathbf{L}}}\right)  =f\left(  a\right)  =f\left(
\varphi\left(  a\right)  \right)  =\varphi\left(  a\right)  ^{\prime
\mathcal{T}_{i}\left(  {\mathbf{L}}\right)  }$. If $a=c=c^{\prime}$, then
$\varphi\left(  a^{\prime{\mathbf{L}}}\right)  =\varphi\left(  c\right)
=c=c^{\prime\mathcal{T}_{1}\left(  {\mathbf{L}}\right)  }=\varphi\left(
a\right)  ^{\prime\mathcal{T}_{1}\left(  {\mathbf{L}}\right)  }$.

Finally, we check that $\varphi\left(  a^{\sim{\mathbf{L}}}\right)
=\varphi\left(  a\right)  ^{\sim\mathcal{T}_{i}\left(  {\mathbf{L}}\right)  }%
$. If $a=0^{{\mathbf{L}}}$, then $\varphi\left(  a^{\sim{\mathbf{L}}}\right)
=1^{{\mathbf{L}}}=f\left(  1\right)  ^{\sim\mathcal{T}_{i}\left(  {\mathbf{L}%
}\right)  }=\varphi\left(  a\right)  ^{\sim\mathcal{T}_{i}\left(  {\mathbf{L}%
}\right)  }$. If $a>0^{{\mathbf{L}}}$, then $\varphi\left(  a^{\sim
{\mathbf{L}}}\right)  =\varphi\left(  0^{{\mathbf{L}}}\right)  =f\left(
1\right)  $. Moreover, $f\left(  1\right)  =a^{\sim\mathcal{T}_{i}\left(
{\mathbf{L}}\right)  }=\varphi\left(  a\right)  ^{\sim\mathcal{T}_{i}\left(
{\mathbf{L}}\right)  }$ if $a\in P\left(  {\mathbf{L}}\right)  $ and, since
$a^{\prime}<1^{{\mathbf{L}}}$, $f\left(  1\right)  =\left(  f\left(
a^{\prime{\mathbf{L}}}\right)  \right)  ^{\sim\mathcal{T}_{i}\left(
{\mathbf{L}}\right)  }=\varphi\left(  a\right)  ^{\sim\mathcal{T}_{i}\left(
{\mathbf{L}}\right)  }$ if $a\in N\left(  \mathbf{L}\right)  ^{+}$.
\end{proof}

From Lemma \ref{gustoso} and Theorem \ref{gazzarra} the following
representation of subdirectly irreducible members of $V(\mathbb{AOL})$
immediately follows:

\begin{corollary}
\label{dentone}If $\mathbf{L}\in V(\mathbb{AOL})$ is a subdirectly irreducible
algebra, then ${\mathbf{L}}$ is isomorphic to $\mathcal{T}_{i}\left(  P\left(
{\mathbf{L}}\right)  \right)  $, for some $i\in\left\{  1,2\right\}  $.
\end{corollary}

\subsection{The distributive subvariety}

In this very short subsection we make a note of a small observation about the
distributive subvariety $\mathbb{V}_{1}$ of $V(\mathbb{AOL})$, parlaying
Corollary \ref{dentone} above into a characterisation of a notable set of
generators for $\mathbb{V}_{1}$. The intrinsic interest of this distributive
subvariety is further reinforced by recalling that the $^{\sim}$-free reducts
of its members are Kleene lattices, whence they belong to the variety of
pseudo-Kleene algebras generated by the $^{\sim}$-free reduct of ${\mathbf{D}%
}_{3}$.

Hereafter, by $\mathbf{2}$ we denote the 2-element chain, considered as a lattice.

\begin{theorem}
$\mathbb{V}_{1}$ is generated by the class $\left\{  \mathcal{T}_{i}\left(
\mathbf{2}^{\kappa}\right)  :i\in\left\{  1,2\right\}  ,\kappa\text{ a
cardinal}\right\}  $.
\end{theorem}

\begin{proof}
Let $\mathbf{L}\in\mathbb{V}_{1}$ be a subdirectly irreducible algebra. By
Corollary \ref{dentone}, ${\mathbf{L}}$ is isomorphic to $\mathcal{T}%
_{i}\left(  {\mathbf{M}}\right)  $, for some $i\in\left\{  1,2\right\}  $ and
for some \emph{distributive} lattice ${\mathbf{M}}$ (since the twist product
construction is easily seen to preserve distributivity). Since the class of
distributive lattices is generated as a quasivariety by the single finite
algebra $\mathbf{2}$, it follows that ${\mathbf{L}}\in I\mathcal{T}_{i}\left(
ISP\left(  {\mathbf{2}}\right)  \right)  $. However, it is readily observed
that the mappings $\mathcal{T}_{i}$, viewed as operators that map classes of
lattices to classes of antiortholattices, commute with the class operators $I$
and $S$ (in their respective signatures). Therefore ${\mathbf{L}}\in
IS\mathcal{T}_{i}\left(  P\left(  {\mathbf{2}}\right)  \right)  $ and thus
$\mathbb{V}_{1}si\subseteq IS\mathcal{T}_{i}\left(  P\left(  {\mathbf{2}%
}\right)  \right)  $, whence
\[
\mathbb{V}_{1}=V\left(  \mathbb{V}_{1}si\right)  \subseteq HSPIS\mathcal{T}%
_{i}\left(  P\left(  {\mathbf{2}}\right)  \right)  \subseteq V\left(
\mathcal{T}_{i}\left(  P\left(  {\mathbf{2}}\right)  \right)  \right)
=V\left(  \left\{  \mathcal{T}_{i}\left(  \mathbf{2}^{\kappa}\right)
:i\in\left\{  1,2\right\}  ,\kappa\text{ a cardinal}\right\}  \right)  .
\]

Conversely, it is clear that every member of $\left\{  \mathcal{T}_{i}\left(
\mathbf{2}^{\kappa}\right)  :i\in\left\{  1,2\right\}  ,\kappa\text{ a
cardinal}\right\}  $ is a distributive antiortholattice.
\end{proof}

\subsection{The distributive subvariety with the strong De Morgan property}

The main result of this subsection is a theorem to the effect that
$\mathbb{V}_{3}$, the distributive subvariety of $V(\mathbb{AOL})$ with the
strong De Morgan property, is generated by the single $5$-element chain
${\mathbf{D}}_{5}$. We will achieve this result via a variant of the
well-known technique deployed by Kalman in \cite{Kalman} in order to show that
the $3$-element Kleene chain generates the variety of Kleene lattices.

For a start, we show that:

\begin{lemma}
\label{katanga}Every PBZ* chain is in $\mathbb{V}_{3}$.
\end{lemma}

\begin{proof}
It will suffice to prove that every PBZ* chain $\mathbf{L}$ is an
antiortholattice and satisfies SDM. As to the former claim, suppose $a\in L$.
Either $0=a\wedge a^{\sim}=a$ or $0=a\wedge a^{\sim}=a^{\sim}$, which proves
our conclusion. For SDM, $\left(  a\wedge0\right)  ^{\sim}=1=a^{\sim}%
\vee1=a^{\sim}\vee0^{\sim}$ for all $a\in A$, so w.l.g. let $a,b\in L$ be such
that $a\wedge b>0$. Since $L$ is a linearly ordered antiortholattice, $\left(
a\wedge b\right)  ^{\sim}=0=a^{\sim}\vee b^{\sim} $.
\end{proof}

We now want to find a suitable set of generators for $\mathbb{V}_{3}$. Let
$\mathbf{L}\in\mathbb{V}_{3}$, and let, for any $p\in L$,
\[
C\left(  p\right)  =\left\{  \left(  x,y\right)  :%
\begin{array}
[c]{c}%
x\wedge p=y\wedge p,x^{\prime}\wedge p=y^{\prime}\wedge p,x^{\sim}\wedge
p=y^{\sim}\wedge p,\\
\square x\wedge p=\square y\wedge p,\Diamond x\wedge p=\Diamond y\wedge
p,\Diamond x^{\prime}\wedge p=\Diamond y^{\prime}\wedge p
\end{array}
\right\}  \text{.}%
\]

Although this definition looks complicated, observe that if $\mathbf{L}%
\in\mathbb{V}_{3}$ is subdirectly irreducible (hence an antiortholattice by
Lemma \ref{gustoso}), all of the conditions it contains except for the first
two ones trivialise for all pairs $\left(  x,y\right)  $ s.t. $0<x,y<1$.

\begin{lemma}
\label{ammarolla}Let $\mathbf{L}\in\mathbb{V}_{3}$ and let $p,q\in L$. Then:
(i) $C\left(  p\right)  $ is a congruence on $\mathbf{L}$; (ii) $C\left(
p\right)  \cap C\left(  q\right)  =C\left(  p\vee q\right)  $; (iii) $C\left(
p\right)  \cap C\left(  p^{\prime}\right)  =\Delta$; (iv) $C\left(  p\right)
=\Delta$ iff $p\in P\left(  {\mathbf{L}}\right)  $.
\end{lemma}

\begin{proof}
(i), (ii): left to the reader.

(iii). Suppose $\left(  x,y\right)  \in C\left(  p\right)  \cap C\left(
p^{\prime}\right)  $. Then in particular $x\wedge p=y\wedge p$ and $x^{\prime
}\wedge p^{\prime}=y^{\prime}\wedge p^{\prime}$, whence $x\vee p=y\vee p$. By
\cite[L. 1.(ii), p. 28]{Birkhoff}, $x=y$.

(iv) By (iii), it suffices to prove that $C\left(  p\right)  \subseteq
C\left(  p^{\prime}\right)  $ iff $p^{\prime}\leq p$. From left to right, it
can be checked that $\left(  p\wedge p^{\prime},p^{\prime}\right)  \in
C\left(  p\right)  \subseteq C\left(  p^{\prime}\right)  $, whence $p\wedge
p^{\prime}=p^{\prime}$. Conversely, suppose $p^{\prime}\leq p$ and $\left(
x,y\right)  \in C\left(  p\right)  $. Since $x\wedge p=y\wedge p$, we have
that%
\[
x\wedge p^{\prime}=x\wedge p^{\prime}\wedge p=y\wedge p^{\prime}\wedge
p=y\wedge p^{\prime}\text{,}%
\]
and similarly for the other conditions.
\end{proof}

\begin{lemma}
\label{intzaras}Let $\mathbf{L}\in\mathbb{V}_{3}$ be a subdirectly irreducible
algebra, and let $a,b\in L$. Then: (i) $P\left(  {\mathbf{L}}\right)  ^{+}$ is
closed under meets; (ii) if $a\in P\left(  {\mathbf{L}}\right)  ^{+},b\in
N\left(  {\mathbf{L}}\right)  $, then $b<a$.
\end{lemma}

\begin{proof}
(i) Suppose $a^{\prime}<a,b^{\prime}<b$, but $a^{\prime}\vee b^{\prime}\nless
a\wedge b$. By Lemma \ref{gustoso}.(ii), $a\wedge b\leq a^{\prime}\vee
b^{\prime}$, and thus by Lemma \ref{ammarolla}.(ii)-(iv) $\Delta=C\left(
a^{\prime}\vee b^{\prime}\right)  =C\left(  a^{\prime}\right)  \cap C\left(
b^{\prime}\right)  $. Thus, $\Delta=C\left(  a^{\prime}\right)  $ or
$\Delta=C\left(  b^{\prime}\right)  $, whence using Lemma \ref{ammarolla}.(iv)
again, either $a\leq a^{\prime}$ or $b\leq b^{\prime}$, a contradiction.

(ii) Suppose $a^{\prime}<a,b\leq b^{\prime}$. By Lemma \ref{gustoso}.(ii),
either $a\wedge b^{\prime}<a^{\prime}\vee b$ or $a^{\prime}\vee b\leq a\wedge
b^{\prime}$. If $a\wedge b^{\prime}<a^{\prime}\vee b$, then by the previous
item%
\[
a\wedge\left(  a^{\prime}\vee b\right)  \leq a\wedge\left(  a^{\prime}\vee
b^{\prime}\right)  =a^{\prime}\vee\left(  a\wedge b^{\prime}\right)
<a\wedge\left(  a^{\prime}\vee b\right)  \text{,}%
\]

a contradiction. So $a^{\prime}\vee b\leq a\wedge b^{\prime}$, whence $b\leq
a^{\prime}\vee b\leq a\wedge b^{\prime}\leq a$. But $a\neq b$, since otherwise
$b\leq b^{\prime}<b$. Our conclusion follows.
\end{proof}

Let $\mathbf{L}\in\mathbb{V}_{3}$ be a subdirectly irreducible algebra. Let
$\sim$ be the equivalence whose cosets are $\left\{  0\right\}  ,\left\{
1\right\}  ,N\left(  \mathbf{L}\right)  ^{+},P\left(  \mathbf{L}\right)  ^{+}$
and the singleton of the fixpoint $c$ (if present). Let moreover, for $p\in L
$:%
\begin{align*}
D\left(  p\right)   &  =\left\{  \left(  x,y\right)  :x\sim y\text{ and
}\left(  x\vee x^{\prime}\right)  \wedge p=\left(  y\vee y^{\prime}\right)
\wedge p\right\}  \text{;}\\
E\left(  p\right)   &  =\left\{  \left(  x,y\right)  :x\sim y\text{ and
}\left(  x\vee x^{\prime}\right)  \vee p=\left(  y\vee y^{\prime}\right)  \vee
p\right\}  \text{.}%
\end{align*}

\begin{proposition}
\label{tomtom}(i) $\sim$ is a congruence; (ii) for all $p\in L$, $D\left(
p\right)  $ is a congruence; (iii) for all $p\in L$, $E\left(  p\right)  $ is
a congruence; (iv) for all $p\in L$, either $D\left(  p\right)  =\Delta$ or
$E\left(  p\right)  =\Delta$.
\end{proposition}

\begin{proof}
(i) $\sim$ preserves meets and Kleene complements by Lemma \ref{intzaras}.
Moreover, the Brouwer complement is preserved because $\mathbf{L}$, by Lemma
\ref{gustoso}.(i), is an antiortholattice.

(ii) It is not hard to show, following in the footsteps of \cite{Kalman}, that
$D\left(  p\right)  $ preserves meets and Kleene complements; thus, it
suffices to show that Brouwer complements are also preserved. If $a\sim b$,
either $a=b=0$ or $a,b>0$. We have to prove that%
\[
\left(  a^{\sim}\vee\lozenge a\right)  \wedge p=\left(  b^{\sim}\vee\lozenge
b\right)  \wedge p\text{,}%
\]
but in both cases the left-hand and the right-hand side reduce to $p$ because
$\mathbf{L}$ is an antiortholattice.

(iii) Similar.

(iv) Suppose $\left(  x,y\right)  \in D\left(  p\right)  \cap E\left(
p\right)  $. Then $\left(  x\vee x^{\prime}\right)  \wedge p=\left(  y\vee
y^{\prime}\right)  \wedge p$ and $\left(  x\vee x^{\prime}\right)  \vee
p=\left(  y\vee y^{\prime}\right)  \vee p$. By virtue of \cite[L. 1.(ii), p.
28]{Birkhoff}, $x\vee x^{\prime}=y\vee y^{\prime}$ and so, since $x\sim y$, we
have that $x=y$. Thus $D\left(  p\right)  \cap E\left(  p\right)  =\Delta$,
whence our conclusion follows since $\mathbf{L}$ is s.i.
\end{proof}

\begin{theorem}
\label{merluzzo}The only nontrivial subdirectly irreducible members of
$\mathbb{V}_{3}$ are $\mathbf{D}_{2},\mathbf{D}_{3},\mathbf{D}_{4}%
,\mathbf{D}_{5}$.
\end{theorem}

\begin{proof}
Let $\mathbf{L}\in\mathbb{V}_{3}$ be a nontrivial subdirectly irreducible
algebra. So $0<1$. If $L$ contains no strictly positive element $c$, either
$\mathbf{L}=\mathbf{D}_{2}$ or $\mathbf{L}=\mathbf{D}_{3}$. Otherwise, we show
that there cannot be three distinct such elements. In fact, suppose the
contrary. Then these elements must form a chain $0<w<y<x<1$, with $w^{\prime
}<w,y^{\prime}<y,x^{\prime}<x$. But it is easy to check that $\left(
x,y\right)  \in D\left(  y\right)  $ and $\left(  w,y\right)  \in E\left(
y\right)  $, contra Proposition \ref{tomtom}.(iv). So $L$ has at most two
strictly positive elements $x,y$, which must be comparable if they exist, and
no element $a$ which is not comparable with $a^{\prime}$ by Lemma
\ref{gustoso}.(ii). If%
\[
0<x^{\prime}<y^{\prime}<y<x<1\text{,}%
\]
then%
\[
\left\{  \left\{  0\right\}  ,\left\{  1\right\}  ,\left\{  x\right\}
,\left\{  x^{\prime}\right\}  ,\left\{  y,y^{\prime}\right\}  \right\}
\]
is a partition of $L$, and it is easily checked that the corresponding
equivalence $F$ is a nonzero congruence (since one of its cosets is not a
singleton). Also, $D\left(  y\right)  \neq\Delta$, because $\left(
x,y\right)  \in D\left(  y\right)  $. However, $D\left(  y\right)  $ does not
identify $y$ and $y^{\prime}$, whence $F\cap D\left(  y\right)  =\Delta$. This
contradicts the subdirect irreducibility of $\mathbf{L}$. So $L$ contains a
chain $0<c<c^{\prime}<1$, while it may or may not contain a fixpoint
$d=d^{\prime}$. If so, $\mathbf{L}=\mathbf{D}_{5}$; if not, $\mathbf{L}%
=\mathbf{D}_{4}$.
\end{proof}

\begin{corollary}
$\mathbb{V}_{3}=V\left(  \mathbf{D}_{5}\right)  $.
\end{corollary}

\begin{proof}
In fact, $\mathbf{D}_{i}\leq\mathbf{D}_{5}$ for $i\leq5$.
\end{proof}

\subsection{A better basis for $V\left(  \mathbf{D}_{3}\right)  $}

Finally, we give a result for the single non-orthomodular cover of the Boolean
variety in {$\mathbf{L}_{\mathbb{PBZL}^{\ast}}$, namely, the variety
}$V\left(  \mathbf{D}_{3}\right)  $ generated by the $3$-element chain
$\mathbf{D}_{3}$. As we already recalled, one of the present writers proved in
\cite{CGP} that $V\left(  \mathbf{D}_{3}\right)  $ is axiomatised relative to
$V\left(  \mathbb{AOL}\right)  $ by the identities DIST, SDM and SK. The aim
of this subsection is showing that this basis is redundant, because%
\[
V\left(  \mathbf{D}_{3}\right)  =Mod\left(  Eq\left(  \mathbb{AOL}\right)
\cup\left\{  \text{SK}\right\}  \right)  .
\]
In other words, SK implies both DIST and SDM in the presence of the
antiortholattice axioms.

\begin{lemma}
\label{freccia}Let $\mathbf{L\in}Mod\left(  Eq\left(  \mathbb{AOL}\right)
\cup\left\{  \text{SK, SDM}\right\}  \right)  $. Then, for any $a,b,c\in L$:
{(i) $a\vee\square b=$}$\left(  {a\vee b}\right)  \wedge\left(  \Diamond
a\vee\square b\right)  ${; (ii) $a\vee\left(  b\wedge c\right)  =a\vee$%
}$\left(  \left(  \Diamond b\vee\square a\right)  {\wedge\left(  {a\vee
b}\right)  {\wedge}c}\right)  ${; (iii)\ $a\vee\left(  b\wedge c\right)
=a\vee\left(  \left(  {a\vee b}\right)  \wedge c\right)  =$; (iv)
$a\wedge\left(  b\vee c\right)  =a\wedge\left(  b\vee\left(  a\wedge c\right)
\right)  $; (v) $a\wedge\left(  b\vee c\right)  =\left(  a\wedge b\right)
\vee\left(  a\wedge c\right)  $.}
\end{lemma}

\begin{proof}
{\ (i)
\[%
\begin{array}
[c]{lll}%
{a\vee\square b} & ={a\vee\square b\vee}\left(  \Diamond a\wedge b\right)  &
\text{(SK)}\\
& ={a\vee}\left(  \left(  {\square b\vee}\Diamond a\right)  \wedge b\right)  &
\text{ \cite[L. 4.1, Prop. 5.2]{GLP1+}}\\
& =\left(  {a\vee b}\right)  \wedge\left(  {a\vee\square b\vee}\Diamond
a\right)  & \text{\cite[L. 5.10.(ix)]{GLP1+}, SDM}\\
& =\left(  {a\vee b}\right)  \wedge\left(  \Diamond a\vee\square b\right)  &
\text{Def. \ref{de:bazar}}%
\end{array}
\]
}

{\ (ii)
\[%
\begin{array}
[c]{lll}%
{a\vee}\left(  \left(  \Diamond b\vee\square a\right)  {\wedge\left(  {a\vee
b}\right)  {\wedge}c}\right)  & ={a\vee}\left(  \left(  {\square}a{\vee
b}\right)  \wedge c\right)  & \text{(i)}\\
& ={a\vee}\left(  {\square}a\wedge c\right)  {\vee}\left(  {b}\wedge c\right)
& \text{ \cite[L. 5.10.(ix)]{GLP1+}}\\
& =\left(  \left(  {a\vee\square}a\right)  \wedge\left(  {a\vee}c\right)
\right)  {\vee}\left(  {b}\wedge c\right)  & \text{\cite[L. 5.10.(ix)]{GLP1+}%
}\\
& ={a\vee\left(  b\wedge c\right)  } & \text{\cite[L. 4.1]{GLP1+}, Abs.}%
\end{array}
\]
}

(iii){%
\[%
\begin{array}
[c]{lll}%
{a\vee\left(  b\wedge c\right)  } & ={a\vee}\left(  \left(  \Diamond
b\vee\square a\right)  {{\wedge}c\wedge\left(  {a\vee b}\right)  }\right)  &
\text{(ii)}\\
& =\left(  {a\vee}\Diamond b\vee\square a\right)  {\wedge}\left(  {a\vee
}\left(  {c\wedge\left(  {a\vee b}\right)  }\right)  \right)  & \text{
\cite[L. 5.10.(ix)]{GLP1+}, SDM}\\
& =a\vee\left(  \left(  {a\vee b}\right)  \wedge c\right)  & \text{\cite[L.
4.1]{GLP1+}}%
\end{array}
\]
}

(iv) From (iii) by duality.

(v){%
\[%
\begin{array}
[c]{lll}%
{a\wedge\left(  b\vee c\right)  } & =a\wedge\left(  c\vee\left(  a\wedge
b\right)  \right)  & \text{(iv)}\\
& =\left(  a\wedge b\right)  \vee\left(  a\wedge\left(  c\vee\left(  a\wedge
b\right)  \right)  \right)  & \text{ Lattice prop.}\\
& =\left(  a\wedge b\right)  \vee\left(  a\wedge c\right)  & \text{(iii)}%
\end{array}
\]
}
\end{proof}

Thus, DIST follows once SDM is satisfied. The next step consists in showing
that SDM directly follows from SK in the presence of the remaining axioms of
$V\left(  \mathbb{AOL}\right)  $.

\begin{lemma}
\label{freccetta}Let $\mathbf{L\in}Mod\left(  Eq\left(  \mathbb{AOL}\right)
\cup\left\{  \text{SK}\right\}  \right)  $. Then, for any $a,b\in L$: {(i) if
}$a\wedge b=0$, then either $a=0$ or $b=0$; (ii) $\left(  a\wedge b\right)
^{\sim}=a^{\sim}\vee b^{\sim}$.
\end{lemma}

\begin{proof}
(i) $\mathbf{L}$ is a subdirect product of antiortholattices in $Mod\left(
Eq\left(  \mathbb{AOL}\right)  \cup\left\{  \text{SK}\right\}  \right)  $ by
Lemma \ref{gustoso}.(i), whence it belongs to the quasivariety generated by
such antiortholattices. Now, suppose ex absurdo that $a\wedge b=0$, but
$a,b>0$. This immediately implies that $a,b<1$. By SK, moreover, we would have
that $a=a\wedge\Diamond b\leq\square a\vee b=b$, whence $a\wedge b=a>0 $, a contradiction.

(ii) By Lemma \ref{gustoso}.(i), we can safely assume that $\mathbf{L}$ is an
antiortholattice. If $a\wedge b>0$, then $a,b>0$, whence $\left(  a\wedge
b\right)  ^{\sim}=0=a^{\sim}\vee b^{\sim}$. If $a\wedge b=0$, then by (i)
either $a=0$ or $b=0$, and thus $\left(  a\wedge b\right)  ^{\sim}=1=a^{\sim
}\vee b^{\sim}$.
\end{proof}

\begin{corollary}
$V\left(  \mathbf{D}_{3}\right)  =Mod\left(  Eq\left(  \mathbb{AOL}\right)
\cup\left\{  \text{SK}\right\}  \right)  $.
\end{corollary}

\begin{proof}
Our result follows from Lemmas \ref{freccia} and \ref{freccetta}.
\end{proof}

\section{Horizontal sums\label{orizzo}}

T{he technique of \emph{horizontal sums}, in itself a crucial tool for lattice
theorists, has been investigated for orthomodular lattices e.g. in
\cite{Chajdor, Hard}. In \cite{GLP1+}, we adapted it to PBZ*-lattices along
the following lines.}

{\ Let $\mathbf{L}$ be an orthomodular lattice and $\mathbf{M}$ be any
PBZ$^{\ast}$-lattice, with $L\cap M=\left\{  0,1\right\}  $. The
\emph{horizontal sum} $\mathbf{L}\boxplus\mathbf{M}$ of $\mathbf{L}%
,\mathbf{M}$ is the algebra%
\[
\left\langle L\cup M,{\wedge}^{\mathbf{L}\boxplus\mathbf{M}}{,\,\vee
^{\mathbf{L}\boxplus\mathbf{M}}{,\,}^{\prime\mathbf{L}\boxplus\mathbf{M}%
}{,\,^{\sim\mathbf{L}\boxplus\mathbf{M}},}0,1}\right\rangle ,
\]
where for $a\in L$, ${a{\,}^{\prime\mathbf{L}\boxplus\mathbf{M}}=a}%
^{{\prime\mathbf{L}}}{{\,,\,a^{\sim\mathbf{L}\boxplus\mathbf{M}}=a}}%
^{\sim\mathbf{L}}$, while for $a\in M$, ${a{\,}^{\prime\mathbf{L}%
\boxplus\mathbf{M}}=a}^{{\prime}\mathbf{M}}{\,}$, ${{\,a^{\sim\mathbf{L}%
\boxplus\mathbf{M}}=a}}^{\sim\mathbf{M}}$. Meets are defined as follows:%
\[
a{\wedge}^{\mathbf{L}\boxplus\mathbf{M}}b=\left\{
\begin{array}
[c]{l}%
a{\wedge}^{\mathbf{L}}b\text{ if }a,b\in L\text{;}\\
a{\wedge}^{\mathbf{M}}b\text{ if }a,b\in M\text{;}\\
0\text{, otherwise.}%
\end{array}
\right.
\]
It is not too hard to show that $\mathbf{L}\boxplus\mathbf{M}$ is a
PBZ$^{\ast}$-lattice, and to acknowledge that our assumption to the effect
that at least one of the summands is orthomodular is needed in order to prove
that $a\wedge a^{\prime}\leq b\vee b^{\prime}$ for all }$a,b\in L\cup M${.
}The above construction, moreover, can be easily generalised to an arbitrary
number of summands $\bigoplus\limits_{i\in I}${$\mathbf{L}_{i}$}, provided
that at most one {$\mathbf{L}_{i}$} fails to be orthomodular.

We now proceed with the next

\begin{definition}
A \emph{block} of a {PBZ$^{\ast}$-lattice} $\mathbf{L}$ is a maximal
subalgebra of $\mathbf{L}$ which is either a Boolean algebra or an
antiortholattice chain. By $B(\mathbf{L})$ we denote the set of all blocks of
$\mathbf{L}$.
\end{definition}

We remark that, if $\mathbf{L}=\bigoplus\limits_{{\mathbf{L}_{i}\in
}B(\mathbf{L})}${$\mathbf{L}_{i}$}, then for every subalgebra $\mathbf{H}%
\leq\mathbf{L}$ we have that $\mathbf{H}=\bigoplus\limits_{{\mathbf{H}_{i}\in
}B(\mathbf{H})}${$\mathbf{H}_{i}$}. We further denote by $\mathbb{HPBZ}^{\ast
}$ the class of subdirectly irreducible {PBZ$^{\ast}$-lattices} which are a
horizontal sum of their own blocks, and by $\mathbb{HPBZ}_{fin}^{\ast}$ the
class of finite members of $\mathbb{HPBZ}^{\ast}$. For $a,b\in L$, moreover we
let%
\[
\gamma\left(  a,b\right)  =\left(  a\vee b\right)  \wedge\left(  a\vee
b{^{\sim}}\right)  \wedge\left(  a{^{\sim}}\vee b\right)  \wedge\left(
a{^{\sim}}\vee b{^{\sim}}\right)
\]
and we write $aCb$ if $\gamma\left(  a,b\right)  =0$ and $a\bar{C}b$ otherwise.

\begin{lemma}
\label{fangala}Let $\mathbf{L}$ be a {PBZ$^{\ast}$-lattice and let }$a\in L${.
If }$a^{\sim}=0$, then for any $b\in L$, we have that $aCb$.
\end{lemma}

\begin{proof}
In fact, in such a case $\gamma\left(  a,b\right)  =\left(  a\vee b\right)
\wedge\left(  a\vee b{^{\sim}}\right)  \wedge b\wedge b{^{\sim}=0}$.
\end{proof}

We now adapt some arguments from \cite{BK71} to our setting. The next lemma
provides necessary and sufficient conditions for a {PBZ$^{\ast}$-lattice to be
representable as }horizontal sum of its blocks.

\begin{lemma}
\label{scucca}For a {PBZ$^{\ast}$-lattice} $\mathbf{L}$ the following are equivalent:

\begin{enumerate}
\item $\mathbf{L}=\bigoplus\limits_{{\mathbf{L}_{i}\in}B(\mathbf{L})}%
${$\mathbf{L}_{i}$};

\item $\mathbf{L}$ satisfies the conditions:

\begin{enumerate}
\item if $a\bar{C}b$ then $a\vee b=1$;

\item if $a,b\notin\{0,1\}$, $a^{\sim}=0$ and $\Diamond b=b$, then $a\vee b=1$;

\item if $a^{\sim}=b^{\sim}=0$, then $a\leq b$ or $b\leq a$;

\item if $a\notin\{0,1\}$, then $a^{\sim}=0$ or $\Diamond a=a$.
\end{enumerate}
\end{enumerate}
\end{lemma}

\begin{proof}
($1\rightarrow2$)

(a) Note that if $a\bar{C}b$, then $a,b$ are sharp elements belonging to
different Boolean blocks, and therefore their join is $1$.

(b) If $a,b\notin\{0,1\}$, $a^{\sim}=0$ and $\Diamond b=b$, then $a$ is
unsharp and $b$ is sharp, whence they belong to different blocks and thus
$a\vee b=1$.

(c) and (d) are obvious.

($2\rightarrow1$) If $a\notin S_{K}(L)$, then by (b) $a^{\sim}=0$ and so by
Lemma \ref{fangala} we have that, for any $b\in L$, $aCb$. Therefore, by
condition (a) and \cite[Lemma 1]{BK71}, the orthomodular subalgebra with
universe $S_{K}(L)$ is a horizontal sum of the Boolean blocks of $\mathbf{L}$.
Consider the set
\begin{equation}
U(L)=\left\{  0\right\}  \cup\{a\in L:a^{\sim}=0\}\text{.} \label{U(L)}%
\end{equation}
By condition (c), $U(L)$ is linearly ordered and therefore it closed under
lattice operations. Now, let $a\in U(L)$ and $0<a<1$, whence $0<a^{\prime}<1
$, and suppose further that $a^{\prime\sim}\not =0$. By condition (d),
$\Diamond a^{\prime}=a^{\prime}$. Thus $a$ is sharp and so $a^{\prime}%
=a^{\sim}=0$, a contradiction. Therefore $U(L)$ is closed under $^{\prime} $,
and then it is the universe of a linearly ordered subalgebra of $\mathbf{L}$
which is an antiortholattice. Finally, the fact that $\mathbf{L}%
=\bigoplus\limits_{{\mathbf{L}_{i}\in}B(\mathbf{L})}${$\mathbf{L}_{i}$}
follows from item (b).
\end{proof}

If $\mathbf{L}=\bigoplus\limits_{{\mathbf{L}_{i}\in}B(\mathbf{L})}%
${$\mathbf{L}_{i}$ and }$\mathbf{L}\notin\mathbb{OML}$, then for exactly one
$i $ we have that {$\mathbf{L}_{i}$ is an antiortholattice chain. We now show
that }$\mathbf{L}$ is s.i. iff such an {$\mathbf{L}_{i}=\mathbf{D}_{j}$, for
some }$3\leq j\leq5$.

\begin{lemma}
\label{lem:hsumsdi} Let $\mathbf{L}=\bigoplus\limits_{{\mathbf{L}_{i}\in
}B(\mathbf{L})}${$\mathbf{L}_{i}$ and }$\mathbf{L}\notin\mathbb{OML}$. Then
$\mathbf{L}$ is subdirectly irreducible PBZ*-lattice iff its unique
non-Boolean summand is either $\mathbf{D}_{3}$ or $\mathbf{D}_{4}$ or
$\mathbf{D}_{5}$.
\end{lemma}

\begin{proof}
Suppose first that {$\mathbf{L}_{i}$ is not a Boolean algebra and that it is
different from $\mathbf{D}_{j}$ for all }$3\leq j\leq5$, {whence }%
$\mathbf{D}_{6}\leq${$\mathbf{L}_{i}$. Therefore, there is a subchain in
$\mathbf{L}_{i}$ having the following form:}%
\[
0<a<b<b^{\prime}<a^{\prime}<1\text{.}%
\]

It is not hard to check that the equivalences%
\begin{align*}
\theta_{1}  &  =\left\{  \left(  x,y\right)  \in L^{2}:b\leq x,y\leq
b^{\prime}\text{ or }x=y\right\} \\
\theta_{2}  &  =\left\{  \left(  x,y\right)  \in L^{2}:a\leq x,y\leq b\text{
or }b^{\prime}\leq x,y\leq a^{\prime}\text{ or }x=y\right\}
\end{align*}
are nonzero congruences on {$\mathbf{L}$, and that }$\theta_{1}\cap\theta
_{2}=\Delta$. Consequently, {$\mathbf{L}$ is not s.i.}

Conversely, it suffices to show that $\mathbf{L}\boxplus${$\mathbf{D}_{j}$ is
subdirectly irreducible for all }$3\leq j\leq5$, and for any orthomodular
lattice $\mathbf{L}$. First, we prove that a congruence $\theta<\nabla$ on
$\mathbf{L}\boxplus${$\mathbf{D}_{j}$ cannot identify distinct elements }$a,b$
of $L$, for suppose otherwise. Then the congruence class of $0$ is nontrivial
\cite[Lemma 4.2]{BH} and there is $0<a<1$ s.t. $a\theta0$ and $a^{\prime
}\theta1$. Let further $b$ be any element of $D${$_{j}$ different from }$0$
and from $1$. We have that%
\[
0=a^{\prime}\wedge b\theta1\wedge b=b=0\vee b\theta a\vee b=1\text{,}%
\]
a contradiction. Next, we prove that $\theta$ cannot identify elements of $L $
with elements of $D${$_{j}$ either. In fact, if }$0<a,b<1$, $a\in L$, $b\in
D${$_{j}$ and }$a\theta b${, then }$0=b^{\sim}\theta a^{\sim}=a^{\prime}$,
contradicting what we have previously established. So, $\theta$ can identify
at most elements of $D${$_{j}$. It follows that }$\mathbf{L}\boxplus
${$\mathbf{D}_{j}$ is subdirectly irreducible if $\mathbf{D}_{j}$ is such, and
this happens for all }$3\leq j\leq5$.
\end{proof}

Let now $H$ be a subclass of $\mathbb{HPBZ}_{fin}^{\ast}$ that satisfies the
following properties:

\begin{description}
\item[(a)] If $\mathbf{H}\in H$ and $\mathbf{L}\in\mathbb{HPBZ}_{fin}^{\ast}$
is a subalgebra of $\mathbf{H}$, then $\mathbf{L}$ belongs to $H$;

\item[(b)] If $\mathbf{L}\boxplus${$\mathbf{D}_{4}$}$\in H$, then
$\mathbf{L}\boxplus${$\mathbf{D}_{3}$} is in $H$.
\end{description}

Let $\mathfrak{L}$ be the complete lattice whose universe is the set of all
classes satisfying Conditions (a) and (b) above, lattice-ordered by inclusion.

\begin{theorem}
\label{thm:L<->H} $\mathfrak{L}$ is isomorphic to the lattice $\mathfrak{H}$
of subvarieties of $V(\mathbb{HPBZ}^{\ast})$.
\end{theorem}

\begin{proof}
Let $\phi:\mathfrak{L}\rightarrow\mathfrak{H}$ and $\psi:\mathfrak{H}%
\rightarrow\mathfrak{L}$ be defined by $\phi(H)=V(H)$ and $\psi(\mathbb{V}%
)=\mathbb{V}\cap\mathbb{HPBZ}_{fin}^{\ast}$, respectively. Clearly
$\phi,\,\psi$ are well-defined isotone mappings. We now show that $\psi
\circ\phi=\mathrm{id}_{\mathfrak{L}}$. It can be readily seen that
$H\subseteq\psi\circ\phi(H)=V(H)\cap\mathbb{HPBZ}_{fin}^{\ast}$. Let
$\mathbf{L}\in V(H)\cap\mathbb{HPBZ}_{fin}^{\ast}$, and suppose that
$\mathbf{L}\not \in H$. Since $\mathbf{L}$ is in $\mathbb{HPBZ}_{fin}^{\ast}$,
$\mathbf{L}$ is subdirectly irreducible, and because it is also in $V(H)$, by
J\'{o}nsson's Lemma $\mathbf{L}$ is in $HSP_{u}(H)$. Let us note that not
being a subalgebra of any $\mathbf{H}$ in $H$ is a first order property.
Furthermore, being a horizontal sum of one's blocks is again a first order
property, by Lemma \ref{scucca}. Then, every algebra in $SP_{u}(H)$ is a
horizontal sum of this sort. If $\mathbf{L}$ has one block, then $\mathbf{L}$
is either the two-element Boolean algebra, or by Lemma \ref{lem:hsumsdi} is
either $\mathbf{D}_{3}$, $\mathbf{D}_{4}$ or $\mathbf{D}_{5}$, and in both
cases it belongs to $H$. Otherwise, by Lemma \ref{lem:hsumsdi} again,
$\mathbf{L}$ is the horizontal sum of its Boolean blocks and {$\mathbf{D}_{j}%
$, for some }$3\leq j\leq5$. Again, $\mathbf{L}$ will be in $H$. Therefore,
$\psi\circ\phi(H)=H$.

As regards the converse, it is evident that if $\mathbb{V}\in\mathfrak{H}$,
then $\mathbb{V}\supseteq\phi\circ\psi(\mathbb{V})=HSP(\mathbb{V}%
\cap\mathbb{HPBZ}_{fin}^{\ast})$. For the converse inclusion, let $\mathbb{V}$
be an element of $\mathfrak{H}$, and $\mathbf{L}$ be subdirectly irreducible
in $\mathbb{V}$. Then, $\mathbf{L}$ is the horizontal sum of its Boolean
blocks and an antiortholattice chain. If $\mathbf{H}$ is a finitely generated
subalgebra of $\mathbf{L}$, then $\mathbf{H}$ is finite because it is the
horizontal sum of its Boolean blocks and an antiortholattice chain, which are
both locally finite. Then, either $\mathbf{H}$ is Boolean, or $\mathbf{H}$ is
in $\mathbb{V}\cap\mathbb{HPBZ}_{fin}^{\ast}$. In both cases, $\mathbf{H}$ is
contained in $HSP(\mathbb{V}\cap\mathbb{HPBZ}_{fin}^{\ast})$. This is the case
for any finitely generated subalgebra of $\mathbf{L}$. Therefore, $\mathbf{L}$
itself belongs to $HSP(\mathbb{V}\cap\mathbb{HPBZ}_{fin}^{\ast})$.
\end{proof}

A direct consequence of Theorem \ref{thm:L<->H} is the following:

\begin{corollary}
\label{cor:fin gen} Every subvariety of $V(\mathbb{HPBZ}^{\ast})$ is generated
by its finite members.
\end{corollary}

\emph{Acknowledgement}. R. Giuntini gratefully acknowledges the support of the
Regione Autonoma Sardegna within the project "Modeling the uncertainty:
quantum theory and imaging processing", LR 7/8/2007, RAS CRP-59872. All
authors gratefully acknowledge the support of the Horizon 2020 program of the
European Commission: SYSMICS project, Proposal Number: 689176, MSCA-RISE-2015,
and the suggestions of two referees.

\end{document}